\documentclass[11pt]{amsart}
\usepackage{amsmath, amssymb, amsthm,verbatim}

\numberwithin{equation}{section}

\newtheorem{thm}[equation]{Theorem}
\newtheorem{prop}[equation]{Proposition}
\newtheorem{lemma}[equation]{Lemma}

\theoremstyle{remark}
\newtheorem*{remark}{Remark}

\newcommand{\F}{\mathbb{F}}

\begin{document}

\title{Some planar monomials in characteristic $2$}
 
\author{Zachary Scherr}
\address{
  Department of Mathematics,
  University of Pennsylvania,
  209 South 33rd Street,
  Philadelphia, PA 19104-6395 USA
}
\email{zscherr@math.upenn.edu}
\urladdr{http://www.math.upenn.edu/$\sim$zscherr/}

\author{Michael E. Zieve}
\address{
  Department of Mathematics,
  University of Michigan,
  530 Church Street,
  Ann Arbor, MI 48109-1043 USA
}
\address{Mathematical Sciences Center, Tsinghua University, Beijing 100084, China}
\email{zieve@umich.edu}
\urladdr{http://www.math.lsa.umich.edu/$\sim$zieve/}

\subjclass[2010]{51E20, 11G20, 11T06}


\thanks{The authors were partially supported by NSF grant DMS-1162181.}
 

\begin{abstract}
Planar functions over finite fields give rise to finite projective planes and other combinatorial objects.  They were originally defined  only in odd characteristic, but recently Zhou introduced a definition in even characteristic which yields similar applications.  In this paper we show that certain functions over $\F_{2^r}$ are planar, which proves
a conjecture of Schmidt and Zhou.
 The key to our proof is a new result about the $\F_{q^3}$-rational points on the curve $x^{q-1}+y^{q-1}=z^{q-1}$.
\end{abstract}

\maketitle


\section{Introduction}
Let $q=p^r$ where $p$ is prime and $r$ is a positive integer.
If $p$ is odd then a \textit{planar function} on $\F_q$ is a function
$F\colon\F_q\to\F_q$ such that, for every $d\in\F_q^*$, the function $c\mapsto F(c+d)-F(c)$ is a bijection on $\F_q$.
Planar functions have been used to construct finite projective planes \cite{DO},  
relative difference sets \cite{GS}, error-correcting codes \cite{CDY}, and
$S$-boxes in block ciphers with optimal resistance to differential
cryptanalysis \cite{NK}.

If $p=2$ then there are no functions $F\colon\F_q\to\F_q$ satisfying the defining property of a planar function, since $0$ and $d$ have the same image as one another under the map $c\mapsto F(c+d)-F(c)$.
Recently Zhou \cite{Zhou} introduced a characteristic $2$ analogue of planar functions, which have the same types of applications as do odd-characteristic planar functions.  These will be the focus of the present paper.  If $p=2$, we say that a function $F\colon\F_q\to\F_q$ is planar if, for every $d\in\F_q^*$, the function $c\mapsto F(c+d)+F(c)+dc$ is a bijection on $\F_q$.  Schmidt and Zhou \cite{SZ,Zhou} showed that any function satisfying this definition can be used to produce a relative difference set with parameters $(q,q,q,1)$, a finite projective plane, and certain codes with unusual properties.
In what follows, whenever we refer to a planar function in characteristic $2$, we mean a function satisfying Zhou's definition.

In the recent paper \cite{SZ}, Schmidt and Zhou studied planar functions in characteristic $2$ which have the form $c\mapsto a c^t$, where $a\in\F_q^*$.  They exhibited two classes of such planar monomials on $\F_q$ where $q=2^r$:
\begin{itemize}
\item $t=2^k$ for any $k\ge 0$, any $r$, and any $a\in\F_q^*$
\item $t=2^k+1$ where $r=2k$ and $a=b^2+b$ for any $b\in\F_{2^k}\setminus\F_2$.
\end{itemize}
They also conjectured \cite[Conj.~7]{SZ} that, when $r=6k$ and $t=4^k+16^k$, there exists $a\in\F_q^*$
such that $c\mapsto a c^t$ is planar on $\F_{2^r}$.  We will prove  this conjecture in the following more precise form:

\begin{thm} \label{main}
For any positive integer $k$, write $Q=4^k$ and $q=Q^3$.
If $a\in\F_q^*$ is a $(Q-1)$-th power but not a $3(Q-1)$-th power, then
the function $c\mapsto a c^{Q^2+Q}$ is planar on $\F_q$.
\end{thm}

Our proof relies on the following result of independent interest:

\begin{thm} \label{fermat}
If $Q>1$ is a power of $2$, and $u,v\in\F_{Q^3}^*$ satisfy
$u^{Q-1}+v^{Q-1}=1$, then $uv$ is a cube in $\F_{Q^3}$.
\end{thm}

We note that, in subsequent work, Voloch and the second author \cite{VZ} have proved an analogous result for odd prime powers.

We conclude this introduction with some general remarks about planar monomials in characteristic $2$.
Since $c\mapsto ac^t$ is planar on $\F_q$ if and only if $c\mapsto ac^{t+q-1}$ is planar, we will restrict
to planar monomials of degree less than $q$.
All known planar monomials over $\F_{2^r}$ of degree between $2$ and $2^r-1$
have degree of the form $2^i+2^j$.
In light of this, we divide 
the classification of planar monomials in characteristic $2$ into two parts:
\begin{itemize}
\item monomials of degree $2^i+2^j$, and
\item monomials of other degrees.
\end{itemize}
We checked via computer that, for $r\le 14$, every planar monomial over $\F_{2^r}$ of degree between $2$ and $2^r-1$ has degree $2^i+2^j$.
We suspect that all planar monomials over $\F_{2^r}$ of degree $2^i+2^j$
with $0\le i<j<r$ are displayed in Theorem~\ref{main} and Proposition~\ref{odd}, but we cannot prove this.  We will present results in this direction in
Section~\ref{secquadratic},
after proving our main results in the next section.


\section{Proof of Theorem~\ref{main}}  \label{secmain}

In this section we prove Theorem~\ref{main}.  Our proof relies on Theorem~\ref{fermat}, which we prove first.

\begin{proof}[Proof of Theorem~\ref{fermat}]
If $Q=2^r$ where $r$ is odd, then $Q^3-1\equiv 1\pmod{3}$,
so every element of $\F_{Q^3}$ is a cube (since the map $c\mapsto c^3$ is a
homomorphism from $\F_{Q^3}^*$ to itself which has trivial kernel).
Thus the result holds in this case, so for the rest of this proof we
will assume that $Q$ is a power of $4$.  Let $\omega$ be a fixed primitive
cube root of unity in $\F_Q$.

Pick $u,v\in\F_{Q^3}^*$ such that
$u^{Q-1}+v^{Q-1}=1$.  
Write $U:=u^{Q-1}$ and $V:=v^{Q-1}$, so that $U+V=1$ and both $U$ and $V$ have order
dividing $Q^2+Q+1$.  We must show that $uv$ is a cube, or equivalently that
$UV$ is a $3(Q-1)$-th power in $\F_{Q^3}^*$.

We first dispense with the case that $U$ is in $\F_Q$.  If $U\in\F_Q$ then
$V=U+1$ is also in $\F_Q$, so both $U$ and $V$ have order dividing $Q-1$.
Thus the orders of $U$ and $V$ divide $Q^2+Q+1-(Q-1)(Q+2)=3$, so $U$ and $V$
are cube roots of unity whose sum is $1$.  It follows that $U$ and $V$ are
distinct primitive cube roots of unity, so their product is $1$, which is indeed
a $3(Q-1)$-th power.

Henceforth we assume that neither $U$ nor $V$ is in $\F_Q$.
Let $F(x)$ be the minimal polynomial of $U$ over $\F_Q$, so
$F(x)=(x+U)(x+U^Q)(x+U^{Q^2})$.  Then $F(0)=U^{Q^2+Q+1}=1$.
Since $F(x+1)$ is a monic irreducible polynomial in $\F_Q[x]$
which has $V$ as a root, it must be the minimal polynomial of $V$ over $\F_Q$.
Its constant term $F(1)$ equals $V^{Q^2+Q+1}=1$.  Therefore $F(x)+1$
is a monic degree-$3$ polynomial in $\F_Q[x]$ whose roots include $0$ and $1$, so
$F(x)+1=x(x+1)(x+b)$ for some $b\in\F_Q$.

Next we determine the minimal polynomial of $UV$ over $\F_Q$.
Note that $UV=U^2+U$.
Since $U$ has degree $3$ over $\F_Q$, we must have $U^2+U\not\in\F_Q$, so
that $U^2+U$ also has degree $3$ over $\F_Q$.  Thus, the minimal polynomial
of $UV$ over $\F_Q$ is the unique monic degree-$3$ polynomial in $\F_Q[x]$
which has $UV$ as a root.  This polynomial is $G(x):=x^3+(b^2+b)x^2+x+1$, since
\begin{align*}
G(U^2+U) &= (U^2+U)^3 + (b^2+b)(U^2+U)^2 + U^2+U+1 \\
  &= U^6 + U^5 + U^4 (b^2+b+1) + U^3 + U^2 (b^2+b+1) + U + 1 \\
  &= (U(U+1)(U+b) + 1) \cdot (U^3+bU^2+(b+1)U+1) \\
  &= F(U) \cdot (U^3+bU^2+(b+1)U+1) \\
  &= 0.
\end{align*}

Since $G(x)$ is the minimal polynomial of $UV$ over $\F_Q$, in particular it is
irreducible.  Thus $G(1)\ne 0$, so $b\notin\{\omega,\omega^2\}$.  We will write down the minimal polynomials for all of the cube roots of $UV$.  We first give three factorizations which might involve polynomials over an extension of $\F_Q$:

\begin{align*}
G(x^3) &= \prod_{e^3=b^2+b+1} (x^3+ex^2+1) \\
&= \prod_{e^3=b+\omega} (x^3+e^2 x^2 + ex + 1) \\
&= \prod_{e^3=b+\omega^2} (x^3+e^2 x^2 + ex + 1).
\end{align*}

We leave the easy verification of these factorizations to the reader; Magma code to verify
them is contained in a comment in the TeX file.
First suppose that one of these three factorizations only involves polynomials in
$\F_Q[x]$.  Since the cubes of the roots of $G(x^3)$ have degree $3$ over $\F_Q$, it follows
that the degree-$3$ polynomials in this factorization are irreducible, so that
any root $d$ of any of these degree-$3$ polynomials must satisfy
$d^{Q^2+Q+1}=1$, and hence must be a $(Q-1)$-th power in $\F_{Q^3}^*$.
Since $G(UV)=0$, this implies that $UV$ is a $3(Q-1)$-th power, as desired.

It remains to show that one of the three factorizations above only involves
polynomials in $\F_Q[x]$.  Equivalently, we must show that one of
$b+\omega$, $b+\omega^2$, and $b^2+b+1$ is a cube in $\F_Q^*$.
If this did not happen then, since $b^2+b+1=(b+\omega)(b+\omega^2)$,
the only possibility is that $b+\omega$ and $b+\omega^2$ are in the same
coset in $\F_Q^*/(\F_Q^*)^3$.  But then $(b+\omega^2)/(b+\omega)$ would be in
$(\F_Q^*)^3$, so we could write $(b+\omega^2)/(b+\omega)=e^3$ with $e\in\F_Q^*$.
Now one can easily verify that $d:=(b^2+b+1)e^2+(b+\omega)^2e+b^2+b$
is a root of $G(x)$ in $\F_Q$ (for instance, see the Magma code in the TeX file),
which is impossible.
\end{proof}

Next we prove a lemma expressing the planarity condition for monomials in a
convenient form.

\begin{lemma} \label{mono}
Pick any positive integer $t$ and any $a\in\F_q^*$, where $q>1$ is a power of $2$.
Then $c\mapsto ac^t$ is a planar function on $\F_q$ if and only if, for each $b\in\F_q^*$,
the function $c\mapsto (c+1)^t+c^t+c a^{-1}b^{t-2}$ is bijective on\/ $\F_q$.
\end{lemma}

\begin{proof}
Planarity asserts that, for each $d\in\F_q^*$, the function $F\colon c\mapsto a(c+d)^t+ac^t+dc$
is bijective on\/ $\F_q$.  Equivalently, $\ell_1\circ f\circ\ell_2$ is bijective, where
$\ell_1\colon c\mapsto c/(a d^t)$ and $\ell_2\colon c\mapsto dc$.  Since
$\ell_1\circ f\circ\ell_2\colon c\mapsto (c+1)^t+c^t+c/(a d^{t-2})$, putting $b=1/d$ yields the
result.
\end{proof}

We now prove Theorem~\ref{main}.

\begin{proof}[Proof of Theorem~\ref{main}]
Let $k$ be a positive integer, and write $Q=4^k$ and $q=Q^3$.  Let $a\in\F_q^*$ be a
$(Q-1)$-th power which is not a $3(Q-1)$-th power.

We first show that no $d,e\in\F_q^*$ satisfy $d^{Q^2-1}+d^{Q-1}=a^{-1} e^{3Q-3}$.
For, suppose $d,e\in\F_q^*$ satisfy this equation.
Since $a$ is a $(Q-1)$-th power, also $d^{Q^2-1}+d^{Q-1}$ is a $(Q-1)$-th
power in $\F_q^*$.  But $d^{Q^2-1}+d^{Q-1}=d^{Q-1}(d^{Q-1}+1)^Q$, so we must have
$d^{Q-1}+1=u^{Q-1}$ for some $u\in\F_q^*$.  Now Theorem~\ref{fermat} implies that
$du$ is a cube in $\F_q^*$, so $d^{Q-1}u^{Q^2-Q}= (du)^{Q-1} (u^{(Q-1)/3})^{3Q-3}$
is a $3(Q-1)$-th power; but this expression equals $a^{-1} e^{3Q-3}$, which contradicts
our hypothesis that $a$ is not a $3(Q-1)$-th power.

Thus, for each $e\in\F_q^*$, the polynomial $x^{Q^2-1}+x^{Q-1}+a^{-1}e^{3Q-3}$
has no roots in $\F_q^*$.  Since the function $c\mapsto c^{Q^2}+c^Q+ca^{-1}e^{3Q-3}$
is a homomorphism from the additive group of $\F_q$ to itself, and the kernel of this
homomorphism is trivial,  the function is a bijection on $\F_q$.
Since $(c+1)^{Q^2+Q}+c^{Q^2+Q}=c^{Q^2}+c^Q+1$, this implies that
$c\mapsto (c+1)^{Q^2+Q}+c^{Q^2+Q}+ca^{-1}e^{3Q-3}$ is a bijection on $\F_q$.
For any $b\in\F_q^*$, put $e=b^{(Q+2)/3}$, so that $e^{3Q-3}=b^{Q^2+Q-2}$.
We have shown that $c\mapsto (c+1)^{Q^2+Q}+c^{Q^2+Q}+ca^{-1}b^{Q^2+Q-2}$ is a bijection
on $\F_q$, which by Lemma~\ref{mono} implies that $c\mapsto ac^{Q^2+Q}$ is planar on $\F_q$.
\end{proof}


\section{Planar monomials of degree $2^i+2^j$}  \label{secquadratic}

We conclude this paper with some general remarks about planar polynomials in
characteristic $2$, and in particular about planar monomials of degree $2^i+2^j$.
Here we say that $F(x)\in\F_q[x]$ is planar if the function $c\mapsto F(c)$ is
a planar function on $\F_q$.  Since every function $\F_q\to\F_q$ is represented
by a polynomial, in particular every planar function is represented by a planar
polynomial.  Moreover, if $F,G\in\F_q[x]$ satisfy $F(x)\equiv G(x)\pmod{x^q-x}$ then
$F(x)$ is planar if and only if $G(x)$ is planar.  Thus, in order to classify
planar polynomials, it suffices to classify their residues mod $(x^q-x)$, or equivalently,
to classify planar polynomials of degree less than $q$.

If $p$ is any (odd or even) prime other than $3$, then (for every $r$) all
known planar polynomials over $\F_{p^r}$ of degree less than $p^r$ have the
property that the degree of every term is the sum of at most two powers of $p$.
The prime $p=3$ must be excluded, due to examples from \cite{CM}.
Planar polynomials in which every term has degree the sum of at most two powers of $p$
are especially interesting, since a classification of such polynomials would be
equivalent to a classification of finite commutative semifields \cite{CH}.
To date, there is no $p$ for which there is even a conjectured classification 
of such planar polynomials.  However, if we restrict to monomials then the situation
becomes more tractable.  It is easy to see that a monomial of degree $p^i$
is planar over $\F_{p^r}$ if and only if $p=2$.  Further, it is well-known (and easy to
prove) that, for any odd prime $p$
and any $0\le i\le j<r$, a monomial over $\F_{p^r}$ of degree $p^i+p^j$ is planar
if and only if $r/\gcd(r,j-i)$ is odd.  Our Theorem~\ref{main} shows that the
analogous assertion for $p=2$ is not true, and the complexity of our proof suggests
that a classification of planar monomials over $\F_{2^r}$ of degree $2^i+2^j$ will
likely be difficult to obtain.  Here we make some remarks about this classification.

We first reformulate the planarity condition for these monomials:

\begin{lemma} \label{mono2}
Fix $0\le i<j<r$ and $a\in\F_{2^r}^*$, and put $q=2^r$ and $G(x):=x^{2^i-1}+x^{2^j-1}$.
The monomial $ax^{2^i+2^j}$ is planar on\/ $\F_q$
if and only if the sets $G(\F_q^*)$ and $a^{-1}(\F_q^*)^{t-2}$ are disjoint.
\end{lemma}

\begin{proof}
Write $q=2^r$ and $t=2^i+2^j$.
By Lemma~\ref{mono}, planarity asserts that, for each $b\in\F_q^*$, the polynomial
$F(x):=(x+1)^t+x^t+xa^{-1}b^{t-2}$ induces a bijection on $\F_q$.  Equivalently,
$F(x)+1$ induces a bijection on $\F_q$, and we compute
\[
F(x)+1 = x^{2^i} + x^{2^j} + x a^{-1} b^{t-2}.
\]
Since this polynomial induces a homomorphism from the additive group of $\F_q$ to itself,
it induces a bijection if and only if it has no nonzero roots in $\F_q$.  
Since $F(x)+1=x(G(x)+a^{-1}b^{t-2})$, the result follows.
\end{proof}

Next we determine all planar monomials of degree $1+2^j$:

\begin{prop} \label{odd}
Pick $0<j<r$ and $a\in\F_{2^r}^*$, and let $T(x):=x^{2^{j-1}}+x^{2^{j-2}}+\dots+x$.
Then $ax^{1+2^j}$ is planar on $\F_{2^r}$
if and only if $j=r/2$ and $T(a^{2^j+1})=0$.
\end{prop}

\begin{proof}
Let $q=2^r$ and $J=\gcd(j,r)$.
By Lemma~\ref{mono2}, $ax^{1+2^j}$ is planar on $\F_q$ if and only if
there do not exist $u,v\in\F_q^*$ such that
\[
1+u^{2^j-1} = a^{-1} v^{2^j-1}.
\]
Since the set of $(2^j-1)$-th powers in $\F_q^*$ is the same as the set of
$\gcd(2^j-1,2^r-1)$-th powers, i.e., the $(2^J-1)$-th powers, planarity asserts
that the equation $1+U^{2^J-1} = a^{-1} V^{2^J-1}$ has no solutions with
$U,V\in\F_q^*$.  Note that the equation $z^{2^J-1}+x^{2^J-1}=a^{-1}y^{2^J-1}$
defines a nonsingular curve $C$ in $\mathbb{P}^2$ of genus
$g:=(2^J-2)(2^J-3)/2$.  Planarity asserts that $C$ has no $\F_q$-rational points
with nonzero coordinates.

First suppose $J\le r/4$.  Then
Weil's bound implies that the number of $\F_q$-rational points on $C$ is at least
\begin{align*}
q+1-2g\sqrt{q} &= q+1-(2^J-2)(2^J-3)\sqrt{q} \\
&\ge q+1-(q^{1/4}-2)(q^{1/4}-3)\sqrt{q} \\
&= 1 + (5q^{1/4}-6)\sqrt{q}.
\end{align*}
At most $3(2^J-1)$ of these points have a coordinate being zero.
Since 
\[
1+(5q^{1/4}-6)\sqrt{q} > 3\sqrt{q} > 3(2^J-1),
\]
there is a point in $C(\F_q)$ with nonzero coordinates, so $ax^{1+2^j}$ is not
 planar on $\F_q$.

Since $J$ is a proper divisor of $r$, the remaining cases are $J=r/2$ and $J=r/3$.
If $J=r/2$ then $j=r/2$, and by \cite[Thm.~1]{Moisio},
there is a point in $C(\F_q)$ with nonzero coordinates if and only if
the canonical additive character $\chi$ on $\F_{2^J}$ satisfies
$\chi(a^{2^J+1})\ne 1$, or equivalently $T(a^{2^J+1})\ne 0$.
If $J=r/3$ then, by \cite[Thm.~2]{Moisio}, the curve $C$ has an $\F_q$-rational
point with nonzero coordinates.
\end{proof}

\begin{remark}
If $j=r/2$ then the condition $T(a^{2^j+1})=0$ can be reformulated as
asserting that $a^{2^j+1}=b^2+b$ for some $b\in\F_{2^j}$.
The number of elements $a\in\F_{2^r}^*$ which satisfy this condition is
$(2^{j-1}-1)(2^j+1)$.  Among these are precisely $2^{j-1}-1$ elements of $\F_{2^j}^*$,
namely, the elements of the form $b^2+b$ with $b\in\F_{2^j}\setminus\F_2$.  Planarity of $ax^{1+2^j}$ for these latter $a$'s was shown in \cite[Thm.~6]{SZ}.
\end{remark}

Since the $(t-2)$-th powers in $\F_q^*$ are precisely the $\gcd(t-2,q-1)$-th powers,
Lemma~\ref{mono2} implies that $ax^{2^i+2^j}$ cannot be planar on $\F_q$
if $\gcd(2^i+2^j-2,q-1)=1$.  This is a special case of \cite[Prop.~9]{SZ}.
It follows that there are no planar monomials over $\F_{2^r}$ of degree $2+2^j$
if $1<j<r$.  But we have not gotten much further in determining which monomials 
of degree $2^i+2^j$ are planar.  In particular, we do not know whether there are any
other planar monomials over $\F_{Q^3}$ of degree $Q+Q^2$, besides the ones
described in Theorem~\ref{main}.  However, we did verify via computer that
the planar monomials in Theorem~\ref{main} and Proposition~\ref{odd} are
the only planar monomials over $\F_{2^r}$ of degree $2^i+2^j$ with $0\le i<j<r\le 50$.


\end{document}